\documentclass[11pt,reqno]{amsart}
\usepackage{fullpage,amsfonts,amsmath,amssymb}
\usepackage{amsfonts,amsmath,amssymb,fullpage,graphicx}
\usepackage{verbatim}
\usepackage{hyperref}

\newcommand{\what}{\widehat}%
\newcommand{\R}{\mathbb R}%
\newcommand{\C}{\mathbb C}%
\newcommand{\Z}{\mathbb Z}%
\newcommand{\N}{\mathbb N}%
\newtheorem{theorem}{Theorem}[section]
\newtheorem{lemma}[theorem]{Lemma}

\theoremstyle{definition}

\theoremstyle{definition}
\newtheorem{remark}[theorem]{Remark}

\numberwithin{equation}{section}

\begin{document}
\title{Wiener Tauberian theorem for rank one semisimple Lie groups}

\author{Sanjoy Pusti and Amit Samanta}

\address[Sanjoy Pusti]{Department of Mathematics and Statistics; Indian Institute of Technology, Kanpur-208016, India.}
\email{spusti@iitk.ac.in}

\address[Amit Samanta]{Department of Mathematics and Statistics; Indian Institute of Technology, Kanpur-208016, India.}
\email{amit.gablu@gmail.com}

\thanks{The second author is financially supported by NBHM, Govt. of India}
\thanks{We are thankful to Professors E. K. Narayanan and Rudra P. Sarkar for their helpful comments and suggestions.
}

\begin{abstract}
We prove a genuine analogue of Wiener Tauberian theorem for $L^1(G//K)$,  where $G$ is a semisimple Lie group of real rank one. This generalizes the corresponding result on the automorphism group of the unit disk by Y. Ben Natan, Y. Benyamini, H. Hedenmalm and Y. Weit (\cite{Ben-2}). 
\end{abstract}

\maketitle
\section{Introduction}
A famous theorem of Norbert Wiener states that for a function $f\in L^1(\R)$, span of translates $f(x-a)$ with complex coefficients is dense in $L^1(\R)$ if and only if the Fourier transform $\what{f}$ is nonvanishing on $\R$. That is the ideal generated by $f$ in $L^1(\R)$ is dense in $L^1(\R)$ if and only if the Fourier transform $\what{f}$ is nonvanishing on $\R$. This theorem is well known as the Wiener Tauberian theorem.  This theorem has been extended to abelian groups. The hypothesis (in the abelian case) is on a Haar integrable function which has nonvanishing Fourier transform on all unitary characters. However, Ehrenpreis and Mautner (in \cite{EM1}) has observed that Wiener Tauberian theorem fails even for the commutative Banach algebra of integrable radial functions
on $\mathrm {SL}(2, \R)$. This failure can  be attributed to the existence of the nonunitary uniformly bounded representations in groups of this class (see \cite{EM2, K-S}). However a modified version of the theorem was established in \cite[Theorem 6]{EM1} for radial functions in $L^1(\mathrm{SL}(2, \R))$. In their theorem they prove that if a function $f$ satisfies ``not-to-rapidly decay'' condition and nonvanishing condition on some extended strip, etc., then the ideal generated by $f$ is dense in $L^1(\mathrm{SL}(2, \R)//\mathrm{SO}(2))$.   This has been  extended to all rank one semisimple Lie groups in the $K$-biinvariant setting (see \cite{Ben-0}, \cite{Sarkar-1998}). For $\delta>0$, let $S_{1, \delta}=\{\lambda\in\C\mid |\Im\lambda|\leq \rho+\delta\}$ and $S_1=\{\lambda\in\C\mid |\Im\lambda|\leq \rho\}$. Then their theorem (for a single function) is as follows:
\begin{theorem}
Let $f\in L^1(G//K)$ be such that its spherical transform $\what{f}$ satisfies the following:
\begin{enumerate}
\item $\what{f}$ is analytic on $S_{1, \delta}^\circ$, continuous on $S_{1, \delta}$,
\item $\lim_{|\lambda|\rightarrow\infty}\what{f}(\lambda)=0$ in $S_{1, \delta}$,
\item $\what{f}(\lambda)\not=0$ for all $\lambda\in S_{1, \delta}$ and 
\item $\limsup_{|t|\rightarrow\infty}\what{f}(t)e^{Ke^{|t|}}>0$ for all $K>0$

\end{enumerate}
Then the ideal generated by $f$ in $L^1(G//K)$ is dense in $L^1(G//K)$.
\end{theorem}
The condition on the extended strip is due to technical reason. With the extended strip condition
the theorem above has been extended to the full group $\mathrm{SL}(2, \R)$ (see \cite{Sarkar-1997}), on rank one symmetric spaces (see \cite{Sarkar-1998}) and on arbitrary rank symmetric space (see \cite{Sitaram-1980}).  See also \cite{Sitaram-1988}, \cite{Naru-2009, Naru-2011} for furthur reference. Y. Ben Natan, Y. Benyamini, H. Hedenmalm and Y. Weit (in \cite{Ben-1, Ben-2}) proved a {\em genuine analogue}  of  the Wiener Tauberian theorem without the extended strip condition on $\mathrm{SL}(2,\R)$ in the $K$-biinvariant setting.  The extra nonvanishing condition on the  extended strip was removed for rank one semisimple Lie groups in \cite{Pusti-2011}.  No other result  is known on  Wiener Tauberian theorem without this superfluous  conditions. Our aim in this paper is to extend this result  to real rank one semisimple Lie group in the $K$-biinvariant setting. Thus we prove a {\em genuine analogue} of Wiener Tauberian theorem on real rank one semisimple Lie group in the $K$-biinvariant setting.

For any function $F$ on $\R$, we let $$\delta_\infty^+(F)=-\limsup_{t\rightarrow\infty} e^{-\frac{\pi}{2\rho}t}\log|F(t)|.$$  Then our theorem states that,

\begin{theorem}\label{main-theorem}
Let $\{f_\alpha\mid \alpha\in \Lambda\}$ be a collection of functions in $L^1(G//K)$, such that $\{\what{f_\alpha}\mid\alpha\in \Lambda\}$ has no common zero in $S_1$ and $\inf_{\alpha\in\Lambda}\delta_\infty^+(\what{f_\alpha})=0$. Then the ideal generated by $\{f_\alpha\mid \alpha\in \Lambda\}$  is dense in in $L^1(G//K)$.
\end{theorem}
The proof of this theorem is an adaptation of the proof of Ben Natan et. al. (\cite{Ben-2}), which uses resolvent tranform method.
The outline of the proof of the theorem is as follows:
Let $I$ be the ideal generated by $\{f_\alpha\mid \alpha\in \Lambda\}$ in $L^1(G//K)$.
\begin{enumerate}
\item First we prove that for each $\lambda\in\C$ with $\Im\lambda>0$, there exists a family of functions $\{b_\lambda\}$ such that $\widehat{b_\lambda}(\xi)=\frac{1}{\xi^2-\lambda^2}$ for all $\xi\in \mathbb{R}$. Also $b_\lambda\in L^1(G//K)$ if and only if $\Im\lambda>\rho$ and $\{b_\lambda\mid \Im\lambda>\rho\}$ span a dense subspace of $L^1(G//K)$.

\item Let $g\in L^\infty(G//K)$ such that $g$ annihilates $I$. We define the resolvent transfrom $\mathcal R[g]$ by $$\mathcal R[g](\lambda)=\langle b_\lambda, g\rangle, \Im\lambda>\rho.$$ Considering $g$ as a bounded linear functional on $L^1(G//K)/I$, we write $$\mathcal R[g](\lambda)=\langle B_\lambda, g\rangle, \Im\lambda>\rho$$ where $B_\lambda=b_\lambda + I\in L^1(G//K)/I, \Im\lambda>\rho$.
\item By the Banach algebra theory (using the fact that spherical transforms of the  elements of $I$ have no common zero), $\lambda\mapsto B_\lambda$ can be extended as a $L^1(G//K)/I$ valued even entire function. This implies that, $\mathcal R[g]$ extends as an even entire function.
\item To get an explicit expression of $\mathcal R[g](\lambda)$ we need  representatives of the cosets $B_\lambda$  for $0<\Im \lambda<\rho$. For $f\in L^1(G//K)$, it will be shown that there is an $L^1$ function $T_\lambda f$ such that $$\what{T_\lambda f}(\xi)=\frac{\what{f}(\lambda)-\what{f}(\xi)}{\xi^2-\lambda^2}, \xi\in\R.$$ If $f\in I$, then $\frac{T_\lambda f}{\what{f}(\lambda)}$ will be a representative of $B_\lambda$, when $\lambda$ is not a zero of $\what{f}$. Since the spherical transforms of the elements of $I$ have no common zeros,  such a representation always exists.

\item We estimate the $L^1$ norm of $b_\lambda$ and $T_\lambda f$, which gives the necessary estimate for $\mathcal R[g]$.  Then a complex analysis technique (\cite{Hedenmalm-1985, Dah}), with the help of this estimate and not-to-rapid-decay condition, it will be shown that $\mathcal R(g)=0$.

\item By denseness of $\{b_\lambda\mid \Im\lambda>\rho\}$,  $g=0$.
\end{enumerate}
 
 As we mentioned earlier that the crux of proof is the resolvent transform method. Beside that both the solutions $\phi_\lambda$ and $\Phi_\lambda$ of the equation $L\phi=-(\lambda^2 + \rho^2)\phi$ play crucial role in the proof ($L$ is the Laplace-Beltrami operator on $G/K$). For $\mathrm{SL}(2, \R)$, they are given by (upto constants) the Legendre functions of first  and second kind respectively: we denote them by  $P_\lambda$ and $Q_\lambda$\footnote{ $P_\lambda, Q_\lambda$ differ from the usual Legendre function by cerain parametrization.} respectively. Whereas in general (rank one case) they are given by the hypergeometric functions (see (\ref{phi}), (\ref{formula-Phi-1}) in the next section).   
 
An integration formula involving the Legendre functions $P_\lambda$ and $Q_\lambda$ (\cite[p. 770, 7.114 (1)]{Grad}, \cite[(2-11)]{Ben-2}) directly gives that $\what{Q_\lambda}(\xi)=\frac{1}{\xi^2-\lambda^2}, \xi\in\R.$Therefore in the $\mathrm{SL}(2, \R)$ case $b_\lambda=Q_\lambda$. But in general, we could not locate similar integral formula involving $\phi_\lambda$ and $\Phi_\lambda$. We overcome this obstacle in the following way : 
Since $\Phi_\lambda$ solves $L\phi=-(\lambda^2 + \rho^2)\phi$ on $G\setminus K$, by a simple ditribution theoretic argument on the group level, we show that $\what{\Phi_\lambda}(\xi)=\frac{k(\lambda)}{\xi^2-\lambda^2}, \xi\in\R$, for some constant $k(\lambda)$. We find the constant $k(\lambda)$ by putting suitable values of $\xi$. Therefore by defining $b_\lambda=\frac{1}{k(\lambda)}\Phi_\lambda$ we get $\what{b_\lambda}(\xi)=\frac{1}{\xi^2-\lambda^2}$, as required.

Steps $(2)$ and $(3)$ (in the outline of the proof) are exactly similar to Ben Natan et. al..

For Step (4), we let $f\in I$. To get a representative of $B_\lambda, 0<\Im\lambda<\rho$, it is necessary to find a function $T_\lambda f\in L^1$ such that $\what{T_\lambda f}(\xi)=\frac{\what{f}(\lambda)-\what{f}(\xi)}{\xi^2-\lambda^2}, \xi\in\R$. In the $\mathrm{SL}(2, \R)$ case, this is achived by defining $T_\lambda f$ as 
\begin{eqnarray}\label{exp-1} 
T_\lambda f(a_t )=Q_{\lambda}(a_t)\int_t^\infty f(a_s) P_\lambda(a_s)2\sinh 2s \,ds - P_\lambda(a_t)\int_t^\infty f(a_s)Q_{\lambda}(a_s)2\sinh 2s \,ds
\end{eqnarray} 
and using some formule involving $P_\lambda$ and $Q_\lambda$ (\cite[(2-9) (2-10)]{Ben-2}). Instead, intially we simply define  $T_\lambda f:=\widehat{f}(\lambda)b_\lambda -f*b_\lambda.$ This is well defined as it will be shown that $b_\lambda$ is a sum of $L^1$ and $L^p$ (for some $p<2$) function. Hence it is straight forward to see that $\what{T_\lambda f}(\xi)=\frac{\what{f}(\lambda)-\what{f}(\xi)}{\xi^2-\lambda^2}, \xi\in\R$. But we must show that $T_\lambda f\in L^1$. For that we  need to express $T_\lambda f$ as in (\ref{exp-1}). Using the following property of $b_\lambda$ (Lemma \ref{lemma-functional equation of b-lambda}): \begin{eqnarray*}\int_K b_\lambda(a_ska_t)dk=
\begin{cases}
b_\lambda(a_t)\phi_\lambda 
(a_s)\hspace{3mm}\textup{if}
\hspace{1mm}t> s\geq 0,\\
b_\lambda(a_s)\phi_\lambda(a_t)\hspace{3mm}\textup{if}
\hspace{1mm}s > t\geq 0.
\end{cases}
\end{eqnarray*}
we show that
\begin{eqnarray}\label{exp-2}
T_\lambda f(a_t)=b_\lambda(a_t)\int_t^\infty f(a_s)\phi_\lambda(a_s)\Delta(s)ds-\phi_\lambda(a_t)
\int_t^\infty f(a_s)b_\lambda(a_s)\Delta(s)ds.
\end{eqnarray}

Using certain properties of hypergeometric functions we estimate $b_\lambda$. Once we have the estimate of $b_\lambda$,  $\|T_\lambda f\|_1$ can be estimated from the formula (\ref{exp-2}), following the similar method in Ben Natan et. al.. Consequently we get the necessary estimate for $\mathcal R[g]$.
In $\mathrm{SL}(2, \R)$ case, $\mathcal R[g]$ satisfies the following estimates: 
$$\left|\mathcal{R}[g](\lambda)\right|\leq C||g||_\infty\frac{1}{d(\lambda,\partial S_1)},\,\,\, |\Im\lambda|>\rho$$
and 
 $$\left|\widehat{f}(\lambda)\mathcal{R}[g](\lambda)\right|\leq C||f||_1||g||_\infty\frac{1}{d(\lambda,\partial S_1)}, \,\,\, |\Im\lambda|<\rho, f\in I$$ where the constant $C$ is independent of $f\in I$. 
Then using a log-log type theorem (\cite[Theorem 5.3]{Ben-2}) it follows that $\mathcal R[g]$ is bounded.
But in general case, an extra polynomial in $\lambda$ appears in the right hand side of the both estimates above. This  difficulty can be removed by a  mild modification of  the log-log type theorem (see Lemma \ref{loglog theorem}). 
That will imply that $\mathcal R[g]$ is a polynomial. Finally it will be proved that $\mathcal R[g]=0$.

\section{Preliminaries}
Most of our notation related to the semisimple
Lie groups and hypergeometric functions is standard and  can be found for example in
\cite{Helga-GGA, GV} and \cite{Koornwinder} respectively. 
We shall follow the standard practice of using the letter $C$ for constants, whose value may change from one line to another. Everywhere in this article the symbol
$f_1\asymp f_2$ for two positive expressions $f_1$ and $f_2$ means that there are positive constants $C_1, C_2$ such
that $C_1f_1\leq f_2\leq C_2f_1$. For a complex number $z$, we will use $\Re z$ and $\Im z$ to
denote respectively the real and imaginary parts of $z$. For $0<p\leq 2$, we let $S_p=\{\lambda\in\C\mid |\Im\lambda|\leq (\frac 2p-1)\rho\}$.

Let $G$ be a connected noncompact semisimple real rank $1$ Lie group with finite centre with Lie algebra $\mathfrak g$.
We fix a Cartan decomposition $\mathfrak g= \mathfrak k+ \mathfrak p$. Let
 $\mathfrak a$ be a maximal abelian subspace of $\mathfrak p$. Since $G$ is of real rank one, we have $\dim\mathfrak a=1$. We denote the real dual of $\mathfrak a$ by  $\mathfrak a^*$. Let $\Sigma\subset \mathfrak a^*$ be the subset of nonzero roots of the pair $(\mathfrak g,\mathfrak a)$. We recall that either $\Sigma=\{-\alpha, \alpha\}$ or $\{-2\alpha, -\alpha, \alpha, 2\alpha\}$ where $\alpha$ is a positive root and the Weyl group $W$ associated  to $\Sigma$ is $\{{\rm Id}, -{\rm Id}\}$ where Id is the identity operator.

 Let $m_1=\dim \mathfrak g_\alpha$ and $m_2=\dim \mathfrak g_{2\alpha}$ where $\mathfrak g_\alpha$ and
 $\mathfrak g_{2\alpha}$ are the root spaces corresponding to $\alpha$ and $2\alpha$. As usual then $\rho=\frac 12(m_1+2m_2)\alpha$ denotes the half sum of the positive roots. Let  $H_0$ be the unique element in $\mathfrak a$ such that $\alpha(H_0)=1$ and through this we identify $\mathfrak a$ with $\R$ as $t\leftrightarrow tH_0$. Then $\mathfrak a_+= \{H\in \mathfrak a\mid \alpha(H)>0\}$ is identified with the set of positive real numbers. We also identify $\mathfrak a^*$ and its complexification $\mathfrak a^*_\C$ with $\R$  and $\C$ respectively by $t\leftrightarrow t\alpha$ and  $z\leftrightarrow z\alpha$, $t\in \R$, $z\in \C$. By abuse of notation we will denote $\rho(H_0)=\frac 12(m_1+2m_2)$ by $\rho$.

 Let $\mathfrak n= \mathfrak g_\alpha+\mathfrak g_{2\alpha}$, $N=\exp \mathfrak n$,
 $K=\exp \mathfrak k$, $A=\exp \mathfrak a$, $A^+=\exp \mathfrak  a_+$ and $\overline{A^+}=\exp \overline{\mathfrak  a_+}$. Then $K$ is a maximal compact subgroup of $G$, $N$ is a nilpotent Lie group and $A$ is a one dimensional vector subgroup identified
  with $\R$.
 Precisely $A$ is parametrized by $a_s=\exp(sH_0)$.  The Lebesgue measure on $\R$ induces the Haar measure on $A$ as $da_s=ds$. Let $M$ be the centralizer of $A$ in $K$. Let  $X=G/K$ be the Riemannian symmetric space of noncompact type associated with the pair $(G, K)$.
Let $\sigma(x)=d(xK, eK)$ where $d$ is the distance function of $X$ induced by the Killing form on $\mathfrak g$.

The group $G$ has  the Iwasawa decomposition
$G=KAN$ and the Cartan decomposition $G=K\overline{A^+}K$.
Using the Iwasawa decomposition  we write an element $x\in G$ as $K(x)\exp H(x)N(x)$.
Let $dg$, $dn$, $dk$ and $dm$ be the Haar measures of $G$, $N$,
$K$ and $M$ respectively where $\int_K\,dk=1$ and $\int_M\,dm=1$.
We have the following integral formulae corresponding to the Cartan
decomposition, which holds for any integrable function:

\begin{equation}
\int_Gf(g)dg=\int_K\int_{\R^+}\int_K f(k_1a_tk_2) \Delta(t)\,dk_1\,dt\,dk_2. \label{polar}
\end{equation} 
where $\Delta(t)=(2\sinh t)^{m_1 + m_2}(2\cosh t)^{m_2}$

A function $f$ is called $K$-biinvariant if $f(k_1xk_2)=f(x)$ for all $x\in G, k_1, k_2\in K$.  We denote the set of all $K$-biinvariant functions in $L^1(G)$ by $L^1(G//K)$.

Let $\mathbb D(G/K)$ be the algebra of $G$-invariant differential operators on $G/K$. The elementary spherical functions $\phi$ are $C^\infty$ functions and are joint eigenfunctions of all $D\in\mathbb D(G/K)$ for some complex eigenvalue $\lambda(D)$. That is $$D\phi=\lambda(D)\phi, D\in\mathbb D(G/K).$$They are parametrized by $\lambda\in\C$. The algebra $\mathbb D(G/K)$ is generated by the Laplace-Beltrami operator $ L$. Then we have, for all $\lambda\in\C, \phi_\lambda$ is a $C^\infty$ solution of  
\begin{equation}\label{phi-lambda}
L\phi=-(\lambda^2 + \rho^2)\phi.
\end{equation}
The $A$-radial part of the Laplace-Beltrami operator is given by 
\begin{equation} 
L_Af(a_t):=\frac{d^2}{dt^2}f (a_t)+\left((m_1+ m_2)\coth t + m_2\tanh t\right)\frac{d}{dt}f(a_t), t>0. 
\end{equation}
Therefore equation (\ref{phi-lambda}) reduces to 
\begin{equation}\label{diff-hyper}
\frac{d^2\phi}{dt^2}+\left((m_1+ m_2)\coth t + m_2\tanh t\right)\frac{d\phi}{dt} + (\lambda^2 + \rho^2)\phi=0, t>0.
\end{equation}
 The change of variable $z:=-\sinh^2 t$ reduces the equation above into the hypergeometric differential equation
\begin{equation}\label{hyper-geo}
z(1-z)\frac{d^2\psi}{dz^2} + [c-(1 + a + b)z]\frac{d\psi}{dz}-abz=0
\end{equation}
with $a=\frac{\rho-i\lambda}{2}, b=\frac{\rho+ i\lambda}{2}, c=\frac{m_1 + m_2 +1}{2}$.
Therefore we have, \begin{eqnarray}\label{phi}\phi_\lambda(a_t)=_2F_1\left(\frac{\rho-i\lambda}{2}, \frac{\rho+ i\lambda}{2}; \frac{m_1 + m_2 +1}{2}; -\sinh^2t\right)
\end{eqnarray}
which is regular at $0$.

Also for $\lambda\not=-i, -2i, \cdots$, another solution $\Phi_\lambda$ of (\ref{diff-hyper}) (or, (\ref{phi-lambda})) on $(0,\infty)$ is given by (see \cite[\S 2.9, (9) (11)]{Erdelyi-1}), 
\begin{eqnarray}\label{formula-Phi-1}
\Phi_\lambda(a_t)&=& (2\cosh t)^{i\lambda-\rho}  \,_2F_1(\frac{\rho-i\lambda}{2}, \frac{m_1 +2}{4} -\frac{i\lambda}{2}; 1-i\lambda; \cosh^{-2} t) \\
&=& (2\sinh t)^{i\lambda-\rho}  \,_2F_1(\frac{\rho-i\lambda}{2}, \frac{-m_1 +2}{4} -\frac{i\lambda}{2}; 1-i\lambda; -\sinh^{-2} t)
\label{formula-Phi-2}
\end{eqnarray}

This solution has singularity at $t=0$. This function $\Phi_\lambda$ has a series representation, called {\em Harish-Chandra series}, for $t>0$. Through the Cartan decomposition we extend $\Phi_\lambda$ as a $K$-biinvariant function on $G\setminus K$. Therefore $\Phi_\lambda$ is a solution of (\ref{phi-lambda}) on $G\setminus K$.

We have for $t\rightarrow\infty$, \begin{equation}\label{est-Phi}
\Phi_\lambda(a_t)=e^{(i\lambda-\rho)t}(1 + O(1)).
\end{equation}
For $\lambda\in \C\setminus i\Z$, $\Phi_\lambda$ and $\Phi_{-\lambda}$ are two linearly independent solutions of (\ref{diff-hyper}). So $\phi_\lambda$ is a linear combination of both $\Phi_\lambda$ and $\Phi_{-\lambda}$. We have $$\phi_\lambda=c(\lambda)\Phi_\lambda + c(-\lambda)\Phi_{-\lambda}$$ where $c(\lambda)$ is the Harish-Chandra $c$-function given by $$c(\lambda)=\frac{2^{\rho-i\lambda} \Gamma(\frac{m_1 + m_2 +1}{2})\Gamma(i\lambda)}{\Gamma(\frac{\rho+ i\lambda}{2}) \Gamma(\frac{m_1+ 2}{4} + \frac{i\lambda}{2})}.$$
It is normalized such that $c(-i\rho)=1$.

Hence, for $\Im\lambda<0$ and as $t\rightarrow\infty$, 
\begin{equation}\label{est-phi}
\phi_\lambda(a_t)=c(\lambda)e^{(i\lambda-\rho)t}(1 + O(1)).
\end{equation}

For any $\lambda\in \C$  the elementary spherical function $\phi_\lambda$ has the following integral representation
$$\phi_\lambda(x)=\int_K e^{-(i\lambda+\rho)H(xk)}\,dk \text{ for all } x\in G.$$ We have the following properties of $\phi_\lambda$:
\begin{enumerate}
\item $\phi_\lambda$ is a $K$-biinvariant function.
\item $\phi_\lambda=\phi_{-\lambda}$, $\phi_\lambda(x)=\phi_\lambda(x^{-1})$. 
\item For fixed $x\in G$, $\lambda\mapsto\phi_\lambda(x)$ is an entire function. 
\item $|\phi_\lambda(x)|\leq 1$ for all $x\in G$ if and only if $\lambda\in S_1$.

\end{enumerate}
The spherical transform $\what{f}$ of a function $f\in L^1(G)$ is defined by the formula
$$\what{f}(\lambda)=\int_Gf(x)\phi_\lambda(x^{-1})\,dx \text{ for all } \lambda\in S_1.$$
Then it follows that $\what{f}$ is  analytic on $S_1^\circ$, continuous on $S_1$. Also $|\what{f}(\lambda)|\rightarrow 0$ for $|\lambda|\rightarrow\infty$ in $S_1$.

Let $C_c^\infty(G//K)$ be the set of all $C^\infty$ compactly supported $K$-biinvariant functions on $G$.  Also let $PW(\C)$ be the set of all entire functions $h:\C\rightarrow \C$ such that for each $N\in \N$, $$\sup_{\lambda\in\C}(1 +|\lambda|)^N |h(\lambda)| e^{-|\Im\lambda|}<\infty$$ and let $PW(\C)_e$ be the set of all even functions in $PW(\C)$. 

Let $\mathcal U(\mathfrak g)$ be the universal
enveloping algebra of $\mathfrak{g}$. For $0<p\leq 2$, let $\mathcal C^p(G//K)$ be the set of all $f\in C^\infty(G//K)$ such that for all  $D_1,D_2\in \mathcal U(\mathfrak g)$, for all $N\in\mathbb N$
$$\sup_{t\geq 0}
|f(D_1;a_t;D_2)|(1+t)^N e^{\frac{2}{p}\rho t}<\infty.$$

Here $f(D_{1};a_{t};D_{2})$ is the usual left and right derivatives of $f$ by $D_1$ and $D_2$ evaluated at $a_t$.
It is topologized by these seminorms and it is a Fr\'{e}chet space.

We define
$\mathcal S(S_p)$ to be be the set of all functions
$h:S_p\rightarrow \mathbb C$ which are continuous on $S_p$,
holomorphic on $S_p^\circ$ (when $p=2$ then the function is simply
$C^\infty$ on $S_2=\R$) and satisfies for all $r,m\in \mathbb N\cup\{0\}$, $$\sup_{\lambda\in
S_p}(1+|\lambda|^{r})|\frac{d^{m}}{d\lambda^{m}}h(\lambda)|<
\infty.$$ Let $\mathcal
S(S_p)_e$  denote the subspaces of
$\mathcal S(S_p)$ consisting of even  functions. Topologized by the seminorms above it  can be
verified that $\mathcal S(S_p), \mathcal S(S_p)_e$  are Fr\'{e}chet spaces.

Then we have the following Paley-Wiener and Schwartz space isomorphism theorems:

\begin{theorem}\label{Paley-Wiener}
The function $f\mapsto \what{f}$ is a topological isomorphism between $C_c^\infty(G//K)$ and $PW(\C)_e$. Also it is a topological isomorphism between $\mathcal C^p(G//K)$ and $\mathcal S(S_p)_e$.
\end{theorem}

\noindent{\bf Hypergeometric function:}  We need the following properties of the hypergeometric functions:
\begin{enumerate}
\item[(a)] The hypergeometric function has the following integral representation for $\Re c>\Re b>0$,
\begin{eqnarray}\label{integral representation of hypergeometric function}
_2F_1(a,b;c;z)=\frac{\Gamma(c)}{\Gamma(b)\Gamma(c-b)}\int_0^1 s^{b-1}(1-s)^{c-b-1}(1-sz)^{-a}ds,\hspace{3mm}|z|<1.
\end{eqnarray}

\item [(b)] \begin{eqnarray}\label{properties of hypergeometric 
function-1}
_2F_1(a,b;c;z)=(1-z)^{-b}{}_2F_1\left(c-a,b;c;\frac{z}{z-1}\right),\hspace{3mm} z\in\mathbb{C}\setminus[1,\infty).
\end{eqnarray}
(see \cite[p. 247, eqn. (9.5.2)]{Lebedev})
\item [(c)] \begin{eqnarray}\label{properties of hypergeometric function-2}
\nonumber c(c+1){}_2F_1(a,b;c;z)&=& c(c-a+1)_2F_1(a,b+1;c+2;z)\\&&+
a\left[c-(c-b)z\right]{}
 _2F_1(a+1,b+1;c+2;z), \hspace{3mm} z\in\mathbb{C}\setminus[1,\infty). 
\end{eqnarray}
(see \cite[p. 240, eqn. (9.1.7)]{Lebedev})

\item[(d)] \begin{eqnarray}\label{properties of hypergermetric function-3}
\nonumber \int_0^1 x^{d-1}(1-x)^{b-d-1}{}_2F_1(a,b;c;x)dx &=&\frac{\Gamma(c)\Gamma(d)\Gamma(b-d)\Gamma(c-a-d)}{\Gamma(b)\Gamma(c-a)\Gamma(c-d)},\\
&&\textup{if}\hspace{1mm}\Re d>0, \Re (b-d)>0,\Re(c-a-d)>0.
\end{eqnarray}
(see \cite[p. 821, 7.512 (3)]{Grad})

\end{enumerate}

\section{The functions $b_\lambda$ : Representatives of $B_\lambda, \Im\lambda>\rho$}
\noindent Let $\mathbb{C_+}=\{z\in\C\mid \Im z>0\}$  be the open upper half plane in $\mathbb{C}$.

For $\lambda\in \mathbb{C}_+$, we define 
\begin{eqnarray}\label{definition of b-lambda}
b_\lambda(a_t):=\frac{i}{2\lambda c(-\lambda)}\Phi_\lambda(a_t), t>0
\end{eqnarray}

where $c$ is the Harish-Chandra $c$-function. We note that $b_\lambda$  is positive when $\lambda=i\eta$ with $\eta>0$. 

\vspace{.5cm}
In this section we show that $b_\lambda$ can be written as a sum of $L^1$ and $L^p (p<2)$ functions. We also show that $b_\lambda\in L^1(G//K)$ if and only if $\Im\lambda>\rho$ (Lemma \ref{lemma-2}), and  estimate their $L^1$-norm (Lemma \ref{lemma- L-1 norm of b-lambda}). We prove that for $\lambda\in\C_+$, $\what{b_\lambda}(\xi)=\frac{1}{\xi^2-\lambda^2}, \xi\in\R$ (Lemma \ref{lemma-spherical transform of b-lambda}).  Finally we prove that $\{b_\lambda\mid\Im\lambda>\rho\}$ is dense in $L^1(G//K)$ (Lemma \ref{denseness of b-lambda}).

\vspace{.5cm}

From the Frobenious method (\cite[Chapter 4, \S 8]{Coddington}) it is known that for each $\lambda\in\C_+$, $b_\lambda(a_t)$ is asymptotically equal to $t^{-(m_1+m_2-1)}$ if $m_1+m_2>1$ and $\log \frac{1}{t}$ if $m_1+m_2=1$ as $t\rightarrow 0+$. But for our purpose we need the estimates also with respect to $\lambda$.

\begin{lemma}\label{lemma- estimates of b lambda}
Let $\lambda\in \mathbb{C}_+$. Then $b_\lambda$ satisfies the following estimates near $0$ and $\infty$.

\noindent{(a)} There is a positive constant $C$ and a natural number $N$ such that for all $t\in(0,1/2]$,
\begin{eqnarray*}
|b_\lambda(a_t)|\leq \begin{cases}C (1+|\lambda|)^Nt^{-(m_1+m_2-1)},\hspace{3mm}\textup{if}\hspace{1mm} m_1+m_2>1\\
C \log \frac{1}{t}\hspace{33.5mm}\textup{if}\hspace{1mm}m_1+m_2=1.
\end{cases}
\end{eqnarray*}
\noindent (b) There is a positive constant $C$ and a natural number $M$ such that for all $t\in [1/2,\infty]$,
$$
|b_\lambda(a_t)|\leq C (1+|\lambda|)^Me^{-(\Im\lambda+\rho)t}.
$$
\end{lemma}
\begin{proof}
(a) Case-1 : Let $m_1+m_2>1$. By (\ref{formula-Phi-2}) and (\ref{properties of hypergeometric 
function-1}),
$$
\Phi_\lambda(a_t)=(2\sinh t)^{i\lambda-\rho}(1+\sinh ^{-2}t)^{\frac{m_1-2}{4}+\frac{i\lambda}{2}}{}_2F_1\left(\frac{-\rho+2}{2}-\frac{i\lambda}{2},\frac{-m_1+2}{4}-\frac{i\lambda}{2};1-i\lambda;\frac{1}{1+\sinh^2 t}\right),t>0.
$$
Now, for $0< t\leq 1/2$,
\begin{eqnarray*}
&&\left|(2\sinh t)^{i\lambda-\rho}(1+\sinh ^{-2}t)^{\frac{m_1-2}{4}+\frac{i\lambda}{2}}\right|\\
&=&  2^{-\Im\lambda-\rho}(\sinh t)^{-(m_1+m_2-1)}(1+\sinh^2t)^{\frac{m_1-2}{4}-\frac{\Im\lambda}{2}},\\
&\leq &C  2^{-\Im\lambda-\rho}(\sinh t)^{-(m_1+m_2-1)}\\
&\leq &C  2^{-\Im\lambda-\rho} t^{-(m_1+m_2-1)}\hspace{2mm}\textup{(since $\sinh t\geq t$)},
\end{eqnarray*}
where $C$ is a constant independent of $\lambda\in\mathbb{C}+$.
Again, letting
$$
a=a(\lambda)=\frac{-\rho+2}{2}-\frac{i\lambda}{2},\hspace{2mm}b=b(\lambda)=\frac{-m_1+2}{4}-\frac{i\lambda}{2},\hspace{2mm}c=c(\lambda)=1-i\lambda.
$$
we see that, for all $\lambda\in \mathbb{C}+$, $a,b,c$ satisfy the following
$$
\Re(c-b)>\frac{m_1+2}{4},\hspace{2mm}\Re(c-a-b)>\frac{m_1+m_2-1}{2},\hspace{2mm}\Re b>\frac{-m_1+2}{4}>-k+\frac{1}{2},
$$
where $k=\left[\frac{m_1}{4}+1\right]$.
Therefore, by Lemma \ref{bound of hypergeometric functon near 1} (in appendix), there is a $2k$-th degree polynomial $P$ of three variables such that 
$$
\left|{}_2F_1\left(a,b;c;\frac{1}{1+\sinh^2 t}\right)\right|\leq \frac{P(|a|,|b|,|c|)}{|(c)_{2k}|}\left|\frac{\Gamma(c+2k)}{\Gamma(b+k)\Gamma(c-b+k)}\right|,\hspace{3mm}t>0.
$$ 
Now it is easy to see that $P(|a|,|b|,|c|)$ is dominated by a $(2k)$-th degree polynomial in $|\lambda|$, where as $(c)_{2k}$ is $2k$-th degree polynomial in $\lambda$ which has no zero in the region $\Im\lambda\geq 0$. Therefore we conclude that
$$
\frac{P(|a|,|b|,|c|)}{|(c)_{2k}|}
$$
is bounded with bound independent of $\lambda\in\mathbb{C}+$. So we have the following estimate for $\Phi_\lambda$ :
$$
|\Phi_\lambda(a_t)|\leq C 2^{-\Im\lambda-\rho} t^{-(m_1+m_2-1)}\left|\frac{\Gamma(2k+1-i\lambda)}{\Gamma(k+\frac{-m_1+2}{4}-\frac{i\lambda}{2})\Gamma(k+\frac{m_1+2}{4}-\frac{i\lambda}{2})}\right|.
$$
Since 
$$
b_\lambda(a_t)=\frac{i}{2\lambda c(-\lambda)}\Phi_\lambda(a_t)= 2^{-\rho-1-i\lambda}\frac{\Gamma(\frac{\rho-i\lambda}{2})\Gamma(\frac{m_1+2}{4}-\frac{i\lambda}{2})}{\Gamma(\frac{m_1+m_2+1}{2})\Gamma(1-i\lambda)}\Phi_\lambda(a_t),
$$
the proof follows from the following fact which can be proved by the Starling approximation formula (see Lemma \ref{gamma-1} in appendix ) : let $A,B>0$ be fixed, then
$
\left|\frac{\Gamma(A+z)}{\Gamma(B+z)}\right|
$
is dominated by a polynomial on the region $\Re z>0$.

Case-2 : Let $m_1+m_2=1$. The proof is similar to the $\mathrm{SL}_2(\R)$ case (see \cite[Lemma 2.3]{Ben-2} ). Using equation (\ref{formula-Phi-1}) and the integral representation of the hypergeometric function (\ref{integral representation of hypergeometric function}), we  write $b_\lambda$, for $\lambda\in\mathbb{C}+$, as
$$
b_\lambda(a_t)=
2^{-2\rho-1
}\int_0^1s^{\frac{m_1+2}{4}-\frac{i\lambda}{2}-1}(1-s)^{\frac{-m_1+2}{4}-\frac{i\lambda}{2}-1}(\cosh^2t-s)^{-\frac{\rho-i\lambda}{2}}ds.
$$
Putting $\cosh^2 t=x+1$, and making the change of variable $s\rightarrow 1-s$, we get 
$$
b_\lambda(a_t)=
2^{-2\rho-1
}\int_0^1 (1-s)^{\frac{m_1+2}{4}-\frac{i\lambda}{2}-1} s^{\frac{-m_1+2}{4}-\frac{i\lambda}{2}-1}(x+s)^{-\frac{\rho-i\lambda}{2}}ds.
$$
Therefore $$|b_\lambda(a_t)|\leq 2^{-2\rho-1
}\int_0^1 (1-s)^{\frac{m_1-2}{4}+\frac{\Im\lambda}{2}} s^{\frac{-m_1-2}{4}+\frac{\Im\lambda}{2}}(x+s)^{-\frac{\rho+\Im\lambda}{2}}ds.$$Now we break the integration into two parts $I_1$ and $I_2$ on $(0, 1/2]$ and $[1/2, 1)$ respectively. It is easy to check that $I_2$ is bounded by a constant independent of $\lambda\in \C_+$.
On the other part, it is easy to check that $$I_1\leq C \int_0^{1/2}s^{\frac{-m_1-2}{4}+\frac{\Im\lambda}{2}}(x+s)^{\frac{m_1-2}{4}-\frac{\Im\lambda}{2}}ds,$$ since $m_1 + m_2=1$. Using integration by parts with $(x+s)^{\frac{m_1-2}{4}-\frac{\Im\lambda}{2}}$ as  first function and $s^{\frac{-m_1-2}{4}+\frac{\Im\lambda}{2}}$ as second function we can deduce that $$I_1\leq C+ C\log\left(1 + \frac{1}{2x}\right).$$ Since $x=\sinh^2 t$ the desired estime follows.

(b) We have $$b_\lambda(a_t)=\frac{i}{2\lambda c(-\lambda)}\Phi_\lambda(t)=\frac{2^{-\rho-i\lambda-1}}{\Gamma(\frac{m_1 + m_2 +1}{2})}\frac{\Gamma(\frac{\rho-i\lambda}{2})\Gamma(\frac{m_1+2}{4}-\frac{i\lambda}{2})}{\Gamma(1-i\lambda)}\Phi_\lambda(a_t).$$Then using Lemma \ref{estimate-c-function} (in appendix) and the asymptotic behavior of $\Phi_\lambda$ (see equation (\ref{est-Phi})) we have, $$|b_\lambda(a_t)| \leq C(1 + |\lambda|)^Me^{-(\rho+\Im\lambda)t}$$ for some $M>0$ .
\end{proof}

Using Lemma \ref{lemma- estimates of b lambda} and the fact that  $\Delta(t)\asymp t^{m_1+m_2}$ near $0$ and $\Delta(t)\asymp e^{2\rho t}$ near $\infty$, we have the following  Lemma:
\begin{lemma}\label{lemma-2}
\begin{enumerate}
\item [(a)] For all $\lambda\in\C_+$, $b_\lambda$ is locally integrable at $e$.
\item [(b)] For $\Im \lambda>\rho$, $b_\lambda\in L^1(G//K)$.
\item[(c)]For all $\lambda\in\C_+$ $b_\lambda$ is in $L^2$ outside neighbourhood of $e$. 
\item[(d)]For each $\lambda\in\C_+$, there exists $p<2$ (depending on $\lambda$) such that $b_\lambda$ is in $L^p$ outside neighbourhood of $e$. 
\end{enumerate}
\end{lemma}

Let $\lambda\in\C_+$. Then it follows that $b_\lambda$ is a $L^2$-tempered $K$-biinvariant distribution. Also we note that $b_\lambda$ can be written as a sum of $L^1$ and $L^p$ ($p<2$) function on $G$. Therefore its spherical transform is a continuous function on $\R$, vanishing at infinity.

Next we calculate the spherical transform of $b_\lambda$. For that we need the following lemma:

\begin{lemma}\label{integral 1}
If $\Im\lambda>\rho$, then $\int_0^\infty\Phi_\lambda(a_t)\Delta(t)dt=
\frac{2i\lambda c(-\lambda)}{\rho^2+\lambda^2}$.
\end{lemma}
\begin{proof}
\begin{eqnarray*}
&&\int_0^\infty \Phi_\lambda(a_t)\Delta(t)dt \\
&=&\int_0^\infty(2\cosh t)^{i\lambda-\rho}{}_2F_1\left(\frac{\rho-i\lambda}{2},\frac{m_1+2}{4}-\frac{i\lambda}{2};1-i\lambda;\cosh^{-2}t\right)(2\sinh t)^{m_1+m_2}(2\cosh t)^{m_2}dt\\
&=& 2^{\frac{i\lambda+\rho-2}{2}}\int_1^\infty{}_2F_1\left(\frac{\rho-i\lambda}{2},\frac{m_1+2}{4}-\frac{i\lambda}{2};1-i\lambda;\frac{2}{y+1}\right)(y+1)^{\frac{i\lambda-\rho+m_2-1}{2}}(y-1)^{\frac{m_1+m_2-1}{2}}dy\\
&&\textup{( by the change of variable $y=\cosh 2t$) }\\
&=& 2^{i\lambda+\rho-1}\int_0^1 {}_2F_1\left(\frac{\rho-i\lambda}{2},\frac{m_1+2}{4}-\frac{i\lambda}{2};1-i\lambda;x\right)x^{-\frac{i\lambda+\rho+2}{2}}(1-x)^{\frac{m_1+m_2-1}{2}}dx\\
&&\textup{( by the change of variable $x=\frac{2}{y+1}$) }\\
&=& 2^{i\lambda+\rho-1}\frac{\Gamma(1-i\lambda)\Gamma(\frac{-i\lambda-\rho} {2})\Gamma(\frac{m_1+m_2+1}{2})}{\Gamma(\frac{m_1+2}{4}-\frac{i\lambda}{2})\Gamma(1+\frac{-i\lambda-\rho}{2})\Gamma(1+\frac{-i\lambda+\rho}{2})}\\
&&\textup{( by}\hspace{1mm}\ref{properties of hypergermetric function-3}\hspace{1mm} \textup{ with}\hspace{1mm} d=\frac{-i\lambda-\rho}{2}, a=\frac{\rho-i\lambda}{2},b=\frac{m_1+2}{4}-\frac{i\lambda}{2},c=1-i\lambda\textup{ )}\\
&=& -\frac{2^{i\lambda+\rho+1}}{\rho^2+\lambda^2}\frac{\Gamma(1-i\lambda)\Gamma(\frac{m_1+m_2+1}{2})}{\Gamma(\frac{m_1+2}{4}-\frac{i\lambda}{2})\Gamma(\frac{-i\lambda+\rho}{2})}=\frac{2i\lambda c(-\lambda)}{\rho^2+\lambda^2}.
\end{eqnarray*}
\end{proof}

\begin{lemma}\label{lemma-spherical transform of b-lambda}
Let $\lambda\in \mathbb{C}_+$. Then $\widehat{b}_\lambda(\xi)=\frac{1}{\xi^2-\lambda^2}$ for all $\xi\in \mathbb{R}$.
\end{lemma}

\begin{proof}
We view $b_\lambda$ as a $K$-invariant function on $G/K$. We define the $K$-invariant distribution $T$ on $G/K$ by 
$$
T:=L b_\lambda+(\rho^2+\lambda^2)b_\lambda.
$$
Since $L b_\lambda(gK)=-(\rho^2+\lambda^2)b_\lambda(gK)$ for all $gK\not=eK$, $T$ must be supported at $\{eK\}$. Therefore we can write 
$$
T=c_0\delta+c_1 L\delta+\cdots+
c_k L^k\delta
$$
for some constants $c_0,c_1,\cdots, c_k$, where $\delta$ denotes the  distribution on $G/K$ defined by $\delta(\phi)=\phi(eK)$ for $\phi\in C_c^\infty(G/K)$. Taking the spherical transform of both side (in the $L^2$-tempered distribution sense), and keeping in mind that $\widehat b_\lambda$ is continuous on $\mathbb{R}$, we get
$$
\widehat b_\lambda(\xi)=\frac{c_0-c_1(\xi^2+\rho^2)+\cdots +c_k(-1)^k(\xi^2+\rho^2)^k }{-\xi^2+\lambda^2} \hspace{3mm}\textup{for all}\hspace{1mm}\xi\in\mathbb{R}.
$$
Since $\widehat b_\lambda$ vanishes at infinity, we must have $c_k=c_{k-1}=\cdots=c_1=0$. Since the constant $c_0$ may depend on $\lambda$, we write $\kappa(\lambda)$ instead of $c_0$. So we have
\begin{eqnarray}\label{1}
\widehat{b}_\lambda(\xi)=\frac{\kappa(\lambda)}{-\xi^2+\lambda^2}\hspace{3mm}\textup{for all}\hspace{1mm}\xi\in\mathbb{R}.
\end{eqnarray}
Therefore, in order to complete the proof we only need to show that $\kappa(\lambda)=-1$. 
First we assume that $\Im\lambda>\rho$. 
Since, under this assumption, $b_\lambda$ is in $L^1$, its spherical transform is a well-defined continuous function on the strip $ S_1$ which is holomorphic in its interior. Therefore, by analytic continuation, we can write (for $\Im\lambda>\rho$)
\begin{eqnarray}
\widehat{b}_\lambda(z)=\frac{\kappa(\lambda)}{-z^2+\lambda^2}\hspace{3mm}\textup{for all}\hspace{1mm}z\in S_1.
\end{eqnarray}
Putting $z=i\rho$ in the above equation we get 
$$
\kappa(\lambda)=(\rho^2+\lambda^2)\int_0^\infty b_\lambda(a_t)\Delta(t)dt\hspace{3mm}\big(\textup{as}\hspace{1mm}\phi_{i\rho}(a_t)\equiv 1\big)
$$ 
which is equal to $-1$ by the definition of $b_\lambda$ (\ref{definition of b-lambda}) and Lemma \ref{integral 1}. Now we drop the restriction on $\lambda$ and assume that $\lambda\in \mathbb{C}_+$. Putting $\xi=0$ in equation \ref{1}, we get 
$$
\kappa(\lambda)=\lambda^2\int_0^\infty b_\lambda(a_t)\phi_0(a_t)\Delta(t)dt.
$$
Now, it is well known that for each fixed $t>0$, $\lambda\rightarrow\Phi_\lambda(t)$ is holomorphic on $\mathbb{C}_+$ so that the same is true for $b_\lambda(a_t)$. Since $b_\lambda$ has the estimates as in Lemma \ref{lemma- estimates of b lambda} and $\phi_0$ satisfies the estimate $|\phi_0(a_t)|\leq C(1+t)e^{-\rho t}$, $t\geq 0$, the integral converges absolutely and an application of Morera's theorem gives that $\kappa(\lambda)$ is holomorphic on $\mathbb{C}_+$. But we have already proved that $\kappa(\lambda)=-1$ if $\Im\lambda>\rho$. Therefore, by analytic continuation, $\kappa(\lambda)=-1$ for all $\lambda\in \mathbb{C}_+$, as required to prove.
\end{proof}

\begin{lemma}\label{lemma- L-1 norm of b-lambda}

\noindent (a) If $\Im\lambda>\rho$,
$
||b_\lambda||_1\leq C\frac{(1+|\lambda|)^K}{\Im\lambda-\rho}
$ for some $C>0$ and non-negative integer $K$.

(b) $||b_\lambda||_1\rightarrow 0$ if $\lambda\rightarrow \infty$ along the positive imaginary axis.
\end{lemma}
\begin{proof}
(a)
Since $\Delta(t)\asymp t^{m_1+m_2}$ near $0$ and $\Delta(t)\asymp e^{2\rho t}$ near $\infty$, (a) and (b) of Lemma \ref{lemma- estimates of b lambda} respectively implies that 
$$
\int_0^{1/2}|b_\lambda(a_t)|\Delta(t)dt\leq C(1+|\lambda|)^N
$$
and 
$$
\int_{1/2}^\infty |b_\lambda(a_t)|\Delta(t)
dt\leq C\frac{(1+|\lambda|)^M}{\Im\lambda-\rho}.
$$
But the right hand side of the first inequality is clearly less than or equal to $C\frac{(1+|\lambda|)^{N+1}}{\Im\lambda-\rho}$. Taking $K$ to be the maximum of $N+1$ and $M$ the lemma follows.

\noindent (b) If $\lambda=i\eta$ ($\eta>0$),  then $b_{i\eta}$ is non negative function. Therefore, if $\eta>\rho$,
 $$
||b_{i\eta}||_1=\int_0^\infty b_{i\eta}(a_t)\Delta(t)dt=\what{b_{i\eta}}(i\rho)=\frac{1}{-\rho^2+\eta^2}\hspace{3mm}\textup{by Theorem}\hspace{1mm}\ref{lemma-spherical transform of b-lambda}.
$$ 
Hence the proof of (b) follows.
\end{proof}

\begin{lemma}\label{denseness of b-lambda}
The functions $\{b_\lambda\mid \Im\lambda>\rho\}$ span a dense subspace of $L^1(G//K)$.
\end{lemma}
\begin{proof}
The proof is similar to the $\mathrm{SL}_2(\mathbb{R})$ case (see \cite[Lemma 3.2]{Ben-2}).   It is enough to show that    $\overline{\text{ span}\{b_\lambda\mid \Im\lambda>\rho\}}$ contains $C_c^\infty(G//K)$.
 
 Let $f\in C_c^\infty(G//K)$. Since $\what{f}$ is entire and it has polynomial decay on any bounded horizontal strip (see Theorem \ref{Paley-Wiener}), Cauchy's formula implies that
$$
\widehat{f}(w)=\frac{1}{2\pi i}\int_{\Gamma_1+\Gamma_2}\frac{\widehat{f}(z)}{z-w}dz, w\in S_1
$$
where $\Gamma_1=\mathbb{R}+i(\rho+ 1)$ rightward and $\Gamma_2=\mathbb{R}-i(\rho+ 1)$ leftward. In the second integral (i.e. integration over $\Gamma_2$), we make the change of variable $z\rightarrow -z$ to get 
\begin{eqnarray*}
\widehat{f}(w)&=&\frac{1}{2\pi i}\int_{\Gamma_1}\frac{\widehat{f}(z)}{z-w}dz+\frac{1}{2\pi i}\int_{\Gamma_1}\frac{\widehat{f}(-z)}{-z-w}(-dz)\\
&=& \frac{1}{2\pi i}\int_{\Gamma_1}\frac{2z\widehat{f}(z)}{z^2-w^2}dz \hspace{3mm}\textup{(since}\hspace{1mm}\widehat{f}\hspace{1mm}\textup{is even) }
\end{eqnarray*}
Now for $z\in\Gamma_1$, $\Im z>\rho$ so (by Lemma \ref{lemma-2}(b)) $b_z$ is in $L^1$ and hence its spherical transform is well-defined continuous function on $ S_1$ and holomorphic in its interior. Therefore, by Lemma \ref{lemma-spherical transform of b-lambda}, $\widehat{b}_z(w)=\frac{1}{z^2-w^2}$. So we can write 
$$
\widehat{f}(w)=\frac{1}{2\pi i}\int_{\Gamma_1}2z\widehat{f}(z)\widehat{b}_z(w)dz.
$$
Lemma \ref{lemma- L-1 norm of b-lambda} together with the decay condition on $\widehat{f}$ imply that the $L^1(G//K)$-valued integral
$$
\frac{1}{2\pi i}\int_{\Gamma_1}2z\widehat{f}(z)b_z(\cdot)dz
$$ converges; and the above equation shows that it must converge to $f$. Since the Riemann sums of the integral are nothing but finite linear combinations of $b_\lambda$'s, we conclude that $f$ is in the closed subspace spanned by $\{b_\lambda\mid \Im\lambda =\rho +1\}$. Hence the Lemma.
\end{proof}

\section{Representatives of $B_\lambda, \,\,\,0<\Im\lambda<\rho$}
Let $f$ be a $K$-biinvariant integrable function on $G$. For each $\lambda$, with $0<\Im\lambda<\rho$, we define
\begin{eqnarray}\label{definition of T-lambda-f}
T_\lambda f:=\widehat{f}(\lambda)b_\lambda -f*b_\lambda.
\end{eqnarray}
Since $b_\lambda$ can be written as a sum of $L^1$ and $L^p$ ($p<2$) function, $T_\lambda f$ is well-defined; in fact it also has the same form i.e. can be written as a sum of $L^1$ and $L^p$ function. In particular its spherical transform is a continuous function on $\mathbb{R}$. The following lemma is an easy consequence of Lemma \ref{lemma-spherical transform of b-lambda}. 

\begin{lemma}\label{spherical transform of T-lambda-f}
Let $0<\Im\lambda<\rho$ and $f$ be a $K$-biinvariant integrable function on $G$ . Then
$$
\widehat{T_\lambda f}(\xi)=\frac{\widehat{f}(\lambda)-\widehat{f}(\xi)}{\xi^2-\lambda^2},\hspace{3mm}\textup{for all}\hspace{1mm} \xi\in\mathbb{R}.
$$
\end{lemma}

Next we  show that $T_\lambda f$ is integrable and estimate its $L^1$ norm. But, for this first we formulate $T_\lambda f$ in another way using the following lemma.

\begin{lemma}\label{lemma-functional equation of b-lambda}
Let $\lambda\in\mathbb{C}_+$. Then
\begin{eqnarray*}\int_K b_\lambda(a_ska_t)dk=
\begin{cases}
b_\lambda(a_t)\phi_\lambda 
(a_s)\hspace{3mm}\textup{if}
\hspace{1mm}t> s\geq 0,\\
b_\lambda(a_s)\phi_\lambda(a_t)\hspace{3mm}\textup{if}
\hspace{1mm}s > t\geq 0.
\end{cases}
\end{eqnarray*}
\end{lemma}

\begin{proof}
First we note that $s\not=t$ imples that $a_ska_t\not \in K$.
Therefore the integral is well defined whenever $s\neq t$, because $b_\lambda$ is smooth outside $K$. Since $b_\lambda$ is $K$-biinvariant, $\int_K b_\lambda(a_ska_t)dk=\int_K b_\lambda(a_tka_s)dk$. We prove the second case. Fix $s>0$. Since $b_\lambda$ is a smooth $K$-biinvariant eigenfunction of $L$ on $G\setminus K$ with eigenvalue $-(\rho^2+\lambda^2)$, the function $g\mapsto\int_K b_\lambda(a_skg)dk$ is a smooth $K$-biinvariant eigenfunction of $L$ on the open ball $B_s$ with the same eigenvalue. Here $B_s=\{k_1a_r k_2\in K \overline{A^+} K\mid r<s\}$. Therefore the function $t\mapsto\int_Kb_\lambda(a_ska_t)dk$ is solution of the  eqn. (\ref{diff-hyper}) on the interval $(0,s)$ which is regular at $0$. Therefore 
$$
\int_Kb_\lambda(a_ska_t)dk=C\phi_\lambda(a_t)\hspace
{3mm}\textup{for all}\hspace{1mm} 0\leq t<s,
$$
for some constant $C$. Putting $t=0$ or equivalently $a_t=e$ in the above equation we get $C=b_\lambda(a_s)$. Hence the proof.
\end{proof}

\begin{lemma}\label{lemma-another formula of T-lambda-f}
Let $0<\Im\lambda<\rho$ and $f$ be a $K$-biinvariant integrable function on $G$. Then for all $t>0$,
$$
T_\lambda f(a_t)=b_\lambda(a_t)\int_t^\infty f(a_s)\phi_\lambda(a_s)\Delta(s)ds-\phi_\lambda(a_t)
\int_t^\infty f(a_s)b_\lambda(a_s)\Delta(s)ds.
$$
\end{lemma}
\begin{proof}
Since 
$$
f*b_\lambda(a_t)=\int_0^\infty f(a_s)\left(\int_K b_\lambda(a_ska_t)dk\right)\Delta(s)ds,
$$
$T_\lambda f(a_t)$ can be written as
\begin{eqnarray*}
T_\lambda f(a_t)=\int_0^\infty f(a_s)\left[b_\lambda(a_t)\phi_\lambda(a_s)-\left(\int_K b_\lambda(a_ska_t)dk\right)\right]\Delta(s)ds.
\end{eqnarray*}
So the proof follows from Lemma \ref{lemma-functional equation of b-lambda}.
\end{proof}

\begin{lemma}\label{L-1 norm of T-lambda-f}
Let $0<\Im\lambda<\rho$ and $f$ be a $K$-biinvariant integrable function on $G$. Also assume that $\lambda\notin B_\rho(0)$. Then 
$T_\lambda f\in L^1(G//K)$ and its $L^1$ norm satisfies $||T_\lambda f||_1\leq C||f||_1(1+|\lambda|)^Ld(\lambda,\partial S_1)^{-1}$, for some non-negative integer $L$, where $d(\lambda,\partial S_1)$ denotes the Euclidean distance of $\lambda$
from the boundary $\partial S_1$ of the strip $ S_1$. 
\end{lemma}

\begin{proof}
We use the formula of $T_\lambda f$ as given in Lemma \ref{lemma-another formula of T-lambda-f}. First we estimate $\int_0^{1/2}|T_\lambda f(a_t)|\Delta(t)dt$. For that we need to rewrite the estimate of $b_\lambda$ given in Lemma \ref{lemma- estimates of b lambda} (a) in the following way : $|b_\lambda(a_t)|\leq r_\lambda(t)$ for all $t\in(0,1/2],$ where
\begin{eqnarray*}
r_\lambda(t)=
\begin{cases}C(1+|\lambda|)^Nt^{-(m_1+m_2-1)},\hspace{3mm}\textup{if}\hspace{1mm} m_1+m_2>1\\
C \log \frac{1}{t}\hspace{33.5mm}\textup{if}\hspace{1mm}m_1+m_2=1.
\end{cases}. 
\end{eqnarray*}
We use the following properties of $r_\lambda$ :

\noindent (i) $r_\lambda$ is a decreasing function,

\noindent (ii) $\int_0^{1/2}r_\lambda(t)\Delta(t)dt\leq C(1+|\lambda|)^N$.

 \noindent Also we use $|b_\lambda(a_t)|\leq C (1+|\lambda|)^Me^{-(\Im\lambda+\rho)t}, t\in [1/2, \infty)$ (see Lemma \ref{lemma- estimates of b lambda}, (b)).

We can write $$\|T_\lambda f\|_1 \leq I_1 + I_2 + I_3 + I_4,$$
where $$I_1=\int_0^{1/2}|b_\lambda(a_t)|\left(\int_t^{\infty}|f(a_s)|\Delta(s)ds\right)\Delta(t)dt $$  $$I_2=\int_0^{1/2}\left(\int_t^\infty|f(a_s)||b_\lambda(a_s)|\Delta(s)ds\right)\Delta(t)dt,$$

$$I_3=\int_{1/2}^\infty|b_\lambda(a_t)|\left(\int_t^\infty |f(a_s)|\,|\phi_\lambda(a_s)|\Delta(s)ds\right)\Delta(t)dt$$ and $$I_4=\int_{1/2}^\infty  |\phi_\lambda(a_t)|\left(\int_t^\infty |f(a_s)|\,|b_\lambda(a_s)|\Delta(s)ds\right)\Delta(t)dt.$$
Then, $$\hspace{-3.15in} I_1\leq\|f\|_1 \int_0^{1/2}r_\lambda(t)\Delta(t) \,dt\leq C(1+|\lambda|)^N \|f\|_1,$$ 

\begin{eqnarray*}
I_2 &=& \int_0^{1/2}\left(\int_t^{1/2}|f(a_s)|\,|b_\lambda(a_s)|\Delta(s)ds\right)\Delta(t)dt  + \int_0^{1/2}\left(\int_{1/2}^\infty|f(a_s)|\,|b_\lambda(a_s)|\Delta(s)ds\right)\Delta(t)dt\\ 
&\leq & \int_0^{1/2}\left(\int_t^{1/2}|f(a_s)| \, r_\lambda(s) \Delta(s)ds\right)\Delta(t)dt  + C \int_{1/2}^\infty|f(a_s)||b_\lambda(a_s)|\Delta(s)ds \\ 
&\leq & \int_0^{1/2}r_\lambda(t) \left(\int_t^{1/2}|f(a_s)| \Delta(s)ds\right)\Delta(t)dt   +C (1 + |\lambda|)^M \int_{1/2}^\infty|f(a_s)| e^{-(\Im\lambda + \rho)s}\Delta(s)ds \\ 
&\leq & C(1 +|\lambda|)^{\max\{M, N\}} \|f\|_1.
\end{eqnarray*}

Using the estimate of $b_\lambda$ and $\phi_\lambda$ we get,
\begin{eqnarray*}
I_3&\leq &  C (1 + |\lambda|)^M |c(-\lambda)|\int_{1/2}^\infty e^{(-\Im\lambda-\rho)t}e^{2\rho t}\left(\int_t^\infty |f(a_s)|e^{(\Im\lambda-\rho)s}\Delta(s)ds\right)dt\\
&=& C (1 + |\lambda|)^M |c(-\lambda)| \int_{1/2}^\infty |f(a_s)|e^{(\Im\lambda-\rho)s} \left(\int_{1/2}^s e^{(-\Im\lambda-\rho)t}e^{2\rho t} dt\right)\Delta(s)ds\\
&=& C (1 + |\lambda|)^M |c(-\lambda)|  \int_{1/2}^\infty |f(a_s)|e^{(\Im\lambda-\rho)s} \left(\frac{e^{(-\Im\lambda+\rho)s}-e^{\frac{(-\Im\lambda+\rho)}{2}}}{\rho-\Im\lambda}\right)\Delta(s)ds\\
&\leq &  \frac{C (1 + |\lambda|)^M |c(-\lambda)| }{\rho-\Im\lambda}||f||_1,\\
&=& C (1 + |\lambda|)^M |c(-\lambda)|\, ||f||_1d(\lambda,\partial S_1)^{-1}.
\end{eqnarray*}
Similarly we can prove that, 
$$
\hspace{-3.3in}  I_4\leq C (1 + |\lambda|)^M |c(-\lambda)|\, ||f||_1d(\lambda,\partial S_1)^{-1}.
$$ 
Since  $$c(-\lambda)=\frac{2^{\rho + i\lambda} \Gamma(\frac{m_1 + m_2 +1}{2})\Gamma(-i\lambda)}{\Gamma(\frac{\rho- i\lambda}{2}) \Gamma(\frac{m_1+ 2}{4} - \frac{i\lambda}{2})}=\frac{2^{\rho + i\lambda} \Gamma(\frac{m_1 + m_2 +1}{2})\Gamma(1-i\lambda)}{(-i\lambda)\Gamma(\frac{\rho- i\lambda}{2}) \Gamma(\frac{m_1+ 2}{4} - \frac{i\lambda}{2})},$$ by Lemma \ref{estimate-c-function} (see appendix), $$|c(-\lambda)|\leq \frac{C}{|\lambda|(1 + |\lambda|)^{(m_1 + m_2 -2)/2}}.$$
Since $\lambda\notin B_\rho(0),$ $|c(-\lambda)|$ is dominated by a polynomial.
Adding the estimates of $I_1, I_2, I_3$ and $I_4$ the desired result follows.
\end{proof}

\begin{remark}
Let $0<\Im\lambda<\rho$. The proof of the above lemma, in fact shows that $T_\lambda f$ always in $L^1$. To get the desired estimate of the $L^1$ norm of $T_\lambda f$ we only need to throw out some neighborhood of $0$. 
\end{remark}

\section{Resolvent transform}
Let $\delta$ be the $K$-biinvariant distribution on $G$ defined by $\delta(\phi)=\phi(e)$ for all $\phi\in C_c^\infty(G//K)$. Also let $L^1_\delta(G//K)$ be the unital Banach algebra generated by $L^1(G//K)$and $\{\delta\}$. Its maximal ideal space is one point compactification $ S_1\cup\{\infty\}$ of $ S_1$, i.e., more precisely, the maximal ideal space is $\big\{L_z:z\in S_1\cup\{\infty\}\big\}$, where $L_z$ is the complex homomorphism on $L^1_\delta(G//K)$ defined by $L_z(f)=\widehat{f}(z)$. Since the spherical transform of elements in $I$ have no common zeros in $ S_1$, it follows that the maximal ideal space of the quotient algebra $L^1_\delta(G//K)/I$ is $\{\infty\}$ i.e. it consists of only one complex homomorphism, namely $f+I\mapsto\widehat{f}(\infty)$. So, by the Banach algebra theory, an element $f+I$ is invertible in $L^1_\delta(G//K)/I$ iff $\widehat{f}(\infty)\neq 0$.

Let $\lambda_0$ be a fixed complex number with $\Im\lambda_0>\rho$. By Lemma \ref{lemma-2}(b),  $b_{\lambda_0}$ is in $L^1$. Therefore, for $\lambda\in \mathbb{C}$, the function $\widehat{\delta}-(\lambda^2-
\lambda_0^2)\widehat{b}_{\lambda_0}$ does not vanish at $\infty$, and hence $\delta-(\lambda^2-\lambda_0^2)b_{\lambda_0}+I$ is inverible in the quotient algebra $L^1_\delta(G//K)/I$. We put
\begin{eqnarray}\label{definition of B-lambda}
B_\lambda=\left(\delta-(\lambda^2-
\lambda_0^2)b_{\lambda_0}+I\right)^{-1}*\left
(b_{\lambda_0}+I\right),\hspace{3mm}\lambda\in\mathbb{C}
\end{eqnarray}
which is, in fact, an element of $L^1(G//K)/I$.
Now, let $g\in L^\infty(G//K)$ annihilates $I$, so that we may consider $g$ as a bounded linear functional on $L^1(G//K)/I$. We define the resolvent tansform $\mathcal{R}[g]$ of $g$ by
\begin{eqnarray}\label{defn-R-g}
\mathcal{R}[g](\lambda)=\left\langle B_\lambda,g\right \rangle
\end{eqnarray}
From (\ref{definition of B-lambda}) it is easy to see that $\lambda\mapsto B_\lambda$ is a Banach space valued even entire function. It follows that $\mathcal{R}[g]$ is an even entire function. The resolvent transform $\mathcal{R}[g]$ has the following properties.
\begin{lemma} \label{properties of resolvent transform}
Assume $g\in L^\infty(G//K)$ annihilates $I$, and fix a function $f\in I$. Let $Z(\widehat{f}):=\{z\in S_1:\widehat{f}(z)=0\}$.

\noindent (a) $\mathcal{R}[g](\lambda)$ is an even entire function. It is given by the following formula :
\begin{eqnarray*}
\mathcal{R}[g](\lambda)=
\begin{cases} 
\langle b_\lambda,g\rangle,\hspace{3mm}\Im \lambda>\rho,\\
\frac{\langle T_\lambda f,g\rangle}{\widehat{f}(\lambda)}, \hspace{3mm}0<\Im\lambda<\rho, \lambda\notin Z(\widehat{f}).
\end{cases}
\end{eqnarray*}
\noindent (b) For $|\Im\lambda|>\rho$, $\left|\mathcal{R}[g](\lambda)\right|\leq C||g||_\infty\frac{(1+|\lambda|)^K}{d(\lambda,\partial S_1)},$

\noindent (c) For $|\Im\lambda|<\rho$, $\left|\widehat{f}(\lambda)\mathcal{R}[g](\lambda)\right|\leq C||f||_1||g||_\infty\frac{(1+|\lambda|)^L}{d(\lambda,\partial S_1)}$, where the constant $C$ is independent of $f\in I$. 
\end{lemma} 
\begin{proof}
(a) Let $\Im\lambda>\rho$. By Lemma \ref{lemma-2}(b) and  Lemma \ref{lemma-spherical transform of b-lambda}, $b_\lambda$ is in $L^1$ and $\widehat{b}_\lambda(z)=\frac{1}{z^2-\lambda^2}$, $z\in S_1$. We observe that for $z\in S_1$,
$$
\frac{1}{\widehat{b}_{\lambda_0}(z)}-\frac{1}{\widehat{b}_\lambda(z)}=\lambda^2-\lambda_0^2
$$
which is equivalent to saying that
$$
\left(1-(\lambda^2-\lambda_0^2)
\widehat{b}_{\lambda_0}(z)\right)\widehat{b}_\lambda(z)=\widehat{b}_{\lambda_0}(z),\hspace{3mm}z\in S_1.
$$
Apply the inverse spherical transform and mod out $I$ to get
$$
\left(\delta-(\lambda^2-\lambda_0^2)
b_{\lambda_0}+I\right)*(b_\lambda+I) =b_{\lambda_0}+I,
$$
Since $\left(\delta-(\lambda^2-\lambda_0^2)
b_{\lambda_0}+I\right)$ is invertible in $L^1_\delta(G//K)/I$, comparing the above equation with \ref{definition of B-lambda} we get $B_\lambda=b_\lambda+I$. Therefore, by the definition of $\mathcal{R}[g](\lambda)$, $\mathcal{R}[g](\lambda)=\langle b_\lambda,g\rangle$.

Next we assume that $0<\Im\lambda<\rho$, $\lambda\notin Z(\widehat{f})$. By Lemma \ref{L-1 norm of T-lambda-f} and Lemma \ref{spherical transform of T-lambda-f}, $T_\lambda f$ is in $L^1$ and $\widehat{T_\lambda f}(z)=\frac{\widehat{f}(\lambda)-\widehat{f}(z)}{z^2-\lambda^2}, z\in S_1$. A small calculation shows that
$$
\left(1-(\lambda^2-\lambda_0^2)
\widehat{b}_{\lambda_0}(z)\right)\frac{\widehat{T_\lambda f}(z)}{\widehat{f}(\lambda)}=\widehat{b}_{\lambda_0}(z)-\frac{\widehat{f}(z)\widehat{b}_{\lambda_0}(z)}{\widehat{f}(\lambda)},\hspace{3mm}z\in S_1.
$$
Again, apply inverse spherical transform and mod out $I$ to get
$$
\left(\delta-(\lambda^2-\lambda_0^2)
b_{\lambda_0}+I\right)*\left(\frac{T_\lambda f}{\widehat{f}(\lambda)}+I\right) =b_{\lambda_0}+I.
$$
Therefore $B_\lambda=\frac{T_\lambda f}{\widehat{f}(\lambda)}+I$ which gives the desired formula for $\mathcal{R}[g](\lambda)$ in this case.
 
(b) It follows from Lemma \ref{lemma- L-1 norm of b-lambda} and the fact that $\mathcal{R}[g](\lambda)$ is even.

(c) From Lemma \ref{L-1 norm of T-lambda-f} it follows that $$\left|\widehat{f}(\lambda)\mathcal{R}[g](\lambda)\right|\leq C||f||_1||g||_\infty\frac{(1+|\lambda|)^L}{d(\lambda,\partial S_1)}$$ for $0<\Im\lambda<\rho,\lambda\not\in B_\rho(0)$, where $C$ is independent of $f\in I$. Since $\widehat{f}(\lambda)\mathcal{R}[g](\lambda)$ is an even continuous function on $ S_1$, the same estimate is true for $|\Im\lambda|<\rho, \lambda\not\in B_\rho(0)$. From (\ref{defn-R-g}) it follows that  $\mathcal{R}[g](\lambda)$ is bounded on $ B_\rho(0)$, with bound independent of $f$. Therefore  on $ B_\rho(0)$$$\left|\widehat{f}(\lambda)\mathcal{R}[g](\lambda)\right|\leq C||f||_1$$ where $C$ is independent of $f$. Hence the proof follows.
\end{proof}

\section{some results from complex analysis}
For any function $F$ on $\R$, we let $$\delta_\infty^+(F)=-\limsup_{t\rightarrow\infty} e^{-\frac{\pi}{2\rho}t}\log|F(t)|$$ and $$\delta_\infty^-(F)=-\limsup_{t\rightarrow\infty} e^{-\frac{\pi}{2\rho}t}\log|F(-t)|.$$
We start with the following Theorem \cite[Theorem 6.8]{Dah} (see also \cite{Hedenmalm-1985}). The proof of the theorem uses the log-log theorem, the Paley-Wiener theorem, Alhfors distortion theorem and the Phragm\'{e}n-Lindel\''{o}f principle. 
\begin{theorem}
Let $M:(0,\infty)\rightarrow (e,\infty)$ be a continuously differentiable decreasing function with 
$$
\lim_{t\rightarrow 0^+}t\log\log M(t)<\infty,\hspace{3mm}\int_0^\infty\log\log M(t)dt<\infty.
$$
Let $\Omega$ be a collection of bounded holomorphic functions on $ S_1^0$ such that 
$$
\inf_{F\in \Omega}\delta^+_\infty(F)=\inf_{F\in \Omega}\delta^-_\infty(F)=0.
$$
Suppose $G$ is a holomorphic function on $\mathbb{C}\setminus Z$ (where $Z$ is a finite subset of $ S_1$) satisfying the following estimates :
\begin{eqnarray*}
|G(z)|&\leq & M\left(d(z,\partial S_1)\right),
\hspace{3mm}z\in\mathbb{C}\setminus S_1,\\
|F(z)G(z)|&\leq & M\left(d(z,\partial S_1)\right),
\hspace{3mm}z\in S_1^0\setminus Z,
\hspace{1mm}\textup{for all}\hspace{1mm}F\in \Omega.
\end{eqnarray*}
Then $G$ is bounded outside a bounded neighborhood of $Z$. 
\end{theorem}

\begin{remark}
(i) The  theorem above is stated in  \cite[Theorem 6.8]{Dah} when $\Omega$ is singleton. But from the proof the statement above  follows. 

(ii) It easy to see that the above theorem remains true if $M$ is continuously differentiable except possibly at finite number of points. 
\end{remark}
For our purpose we need the following theorem which is an easy consequence of the above theorem. 
\begin{theorem}\label{loglog theorem}
Let $M$ and $\Omega$ be as in the previous theorem. Suppose $H$ is an entire function such that, for some non-negative integer $N$, it satisfies the following estimates :
\begin{eqnarray*}
|H(z)|&\leq &(1+|z|)^NM\left(d(z,\partial S_1)\right),
\hspace{3mm}z\in\mathbb{C}\setminus S_1,\\
|F(z)H(z)|&\leq & (1+|z|)^NM\left(d(z,\partial S_1)\right),
\hspace{3mm}z\in S_1^0,
\hspace{1mm}\textup{for all}\hspace{1mm}F\in \Omega.
\end{eqnarray*}
Then $H$ is a polynomial. 
\end{theorem}
\begin{proof}
Define the holomorphic function $G$ on $\mathbb{C}-\{i\rho\}$ by
$$
G(z)=\frac{H(z)}{(z-i\rho)^N}.
$$
Since $|z-i\rho|\geq d(z,\partial S_1)$, it follows that $|G(z)|$ is dominated by some constant times $d(z,\partial S_1)^{-N}M\left(d(z,\partial S_1)\right)$ near $i\rho$ (inside $\mathbb{C}\setminus S_1$); where as it is clearly dominated by constant times $M(d(z,\partial S_1))$ outside a neighborhood of $i\rho$ (inside $\mathbb{C}\setminus S_1$). Similar estimates are true for $F(z)G(z)$. In particular we can write the following estimates :
\begin{eqnarray*}
|G(z)| &\leq &
\begin{cases}
C d(z,\partial S_1)^{-N}M\left(d(z,\partial S_1)\right),\,\,\, z\in\mathbb{C}\setminus S_1\,\,\textup{with}\,\,d(z,\partial S_1)\leq 1\\
C M\left(d(z,\partial S_1)\right),\,\,\hspace{18mm}z\in\mathbb{C}\setminus S_1\,\,\textup{with}\,\,d(z,\partial S_1)>1
\end{cases}\\
|F(z)G(z)| &\leq &
\begin{cases}
C d(z,\partial S_1)^{-N}M\left(d(z,\partial S_1)\right),\,\,\,\,\,\hspace{2mm}z\in S_1^0,\,\,\textup{with}\,\,d(z,\partial S_1)\leq 1\,\,\textup{for all}\,\,F\in\Omega,\\
CM\left(d(z,\partial S_1)\right),\,\,\,\,\,\hspace{19mm}z\in S_1^0,\,\,\textup{with}\,\,d(z,\partial S_1)> 1\,\,\textup{for all}\,\,F\in\Omega,
\end{cases}
\end{eqnarray*}
for some constant $C$ which can be chosen to be $\geq 1$. Define $M^\prime:(0,\infty)\rightarrow (e,\infty)$ by
\begin{eqnarray*}
M^\prime(t)=
\begin{cases}
Ct^{-N}M(t),\,\,\, 0<t\leq 1,\\
CM(t)\,\,\,\, t>1
\end{cases}
\end{eqnarray*}
Note that $M^\prime$ is decreasing and continuously differentiable except possible at the point $1$. 
Now, we observe that, for $0<t\leq 1$,
$$
\log\log(Ct^{-N}M(t))=\log\left(\log(Ct^{-N})+\log M(t)\right)\leq \log(Ct^{-N})+\log\log M(t)
$$
Therefore, it follows that, $M^\prime$ satisfies all the  required properties i.e.
$$
\lim_{t\rightarrow 0^+}t\log\log M^\prime(t)<\infty,\hspace{3mm}\int_0^\infty\log\log M^\prime(t)dt<\infty.
$$ 
Also, by the definition of $M^\prime$, we have 
\begin{eqnarray*}
|G(z)|&\leq &M^\prime\left(d(z,\partial S_1)\right),
\hspace{3mm}z\in\mathbb{C}\setminus S_1,\\
|F(z)G(z)|&\leq & M^\prime\left(d(z,\partial S_1)\right),
\hspace{3mm}z\in S_1^0,
\hspace{1mm}\textup{for all}\hspace{1mm}F\in \Omega.
\end{eqnarray*}
Therefore, by the previous theorem,
$G$ is bounded outside a bounded neighborhood of $i\rho$, and hence $|H(z)|\leq C(1+|z|)^N$ in the same region. But $H$ being entire, by Liouville's theorem, it must be a polynomial.
\end{proof}

\section{Proof of the main theorem}
\begin{proof}[proof of Theorem \ref{main-theorem}:]
Since the ideal generated by $\{f_\alpha\mid \alpha\in\Lambda\}$ is same as the ideal generated by the elements $\left\{\frac{f_\alpha}{||f_\alpha||}\mid \alpha\in\Lambda\right\}$ and $\delta_\infty(\widehat{f})=
\delta_\infty
\left({\widehat{\frac{f}{||f||_1}}}\right)$, we can assume that the functions $f_\alpha$ are of unit $L^1$ norm. Let $g\in L^\infty(G//K)$ annihilates the (closed) ideal $I$ generated by $\{f_\alpha\mid \alpha\in\Lambda\}$. It is enough to show that $g=0$. By the given condition we have
$$
\inf_{\alpha\in\Lambda}\delta^+_\infty(\widehat{f_\alpha})=\inf_{\alpha\in\Lambda}\delta^-_\infty(\widehat{f_\alpha})=0.
$$
 Also, by Lemma \ref{properties of resolvent transform}, the entire function $\mathcal{R}[g]$ satisfies the following estimates
\begin{eqnarray*}
|\mathcal{R}[g](z)|&\leq & C(1+|z|)^N\left(d(z,\partial S_1)\right)^{-1},
\hspace{3mm}z\in\mathbb{C}\setminus S_1,\\
|\widehat{f_\alpha}(z)\mathcal{R}[g](z)|&\leq & C(1+|z|)^N\left(d(z,\partial S_1)\right)^{-1},
\hspace{3mm}z\in S_1^0,
\end{eqnarray*}
for all $\alpha\in\Lambda$, for some constant $C$. (For technical reason the constant $C$ will be taken to be greater than $e$).
 We can define $M:(0,\infty)\rightarrow(e,\infty)$ to be a continuously differentiable decreasing function such that
$M(t)=\frac{C}{t}$ for $0<t< 1$, and  $\int_1^\infty\log\log M(t)dt<\infty$. With this definition of $M$, we  have 
\begin{eqnarray*}
|\mathcal{R}[g](z)|&\leq & (1+|z|)^NM\left(d(z,\partial S_1)\right),
\hspace{3mm}z\in\mathbb{C}\setminus S_1,\\
|\widehat{f_\alpha}(z)\mathcal{R}[g](z)|&\leq & (1+|z|)^NM\left(d(z,\partial S_1)\right),
\hspace{3mm}z\in S_1^0,
\end{eqnarray*}
for all $\alpha\in\Lambda$.
Therefore, by Theorem \ref{loglog theorem}, $\mathcal{R}[g]$ is a polynomial. Since $|\mathcal R[g](z)|\leq \|b_z\|_1 \|g\|_\infty, \Im z>\rho$,  by Lemma \ref{lemma- L-1 norm of b-lambda}, $\mathcal{R}[g](z)\rightarrow 0$ if $|z|\rightarrow\infty$ along the positive imaginary axis. This forces $\mathcal{R}[g]$ to be identically zero. Hence $\langle b_\lambda,g\rangle=0$ whenever $\Im\lambda>\rho$.  By Lemma \ref{denseness of b-lambda},  $\{b_\lambda\mid \Im\lambda>\rho\}$ span a dense subspace of $L^1(G//K)$. Hence $g=0$, as desired. 
\end{proof}

\section{appendix}

\begin{lemma}\label{bound of hypergeometric functon near 1}
Fix two positive numbers $R_1,R_2$, and a non negative integer $k$. Then there is a $2k$th degree polynomial $P$ in three variables with all the co-efficients are non-negative such that
$$
\left|{}_2F_1(a,b;c;x)\right|\leq \frac{P\left(|a|,|b|,|c|\right)}{|(c)_{2k}|}\left|\frac{\Gamma(c+2k)}{\Gamma(b+k)\Gamma(c-b+k)}\right|,\hspace{3mm} 0\leq x<1,
$$
for all $a,b,c$ satisfying the following conditions :
$$
\Re (c-a-b)>R_1, \Re(c-b)>R_2, \Re b>-k+\frac{1}{2}.
$$
\end{lemma}
\begin{proof}
We prove by induction on the non negative integer $k$. So, first assume that $k=0$. By the given condition, $\Re c>\Re b>0$. therefore we can use the integral representation.
$$
\left|_2F_1(a,b;c;x)\right|\leq \left|\frac{\Gamma(c)}{\Gamma(b)\Gamma(c-b)}\right|\int_0^1 s^{\Re b-1}(1-s)^{\Re(c-b-1)}(1-sx)^{-\Re a}ds,\hspace{3mm}0\leq x<1.
$$
Since 
$$
1-s\leq|1-sx|\leq 1,
$$
\begin{eqnarray*}
s^{\Re b-1}(1-s)^{\Re (c-b-1)}(1-sx)^{-\Re a}\leq
\begin{cases}
s^{\Re b-1}(1-s)^{\Re(c-a-b)-1}\leq s^{-\frac{1}{2}}(1-s)^{R_1-1}\hspace{3mm}\textup{if}\hspace{1mm}\Re a>0,\\s^{\Re b-1}(1-s)^{\Re(c-b)-1}\leq s^{-\frac{1}{2}}(1-s)^{R_2-1}\hspace{6.5mm}\textup{if}\hspace{1mm}\Re a<0.
\end{cases}
\end{eqnarray*}
Hence the lemma is true if $k=0$. So assume that the lemma is true for $k=n$ i.e there is a  polynomial $P=P_n$ of degree $2n$ in three variables with all non-negative coefficients such that
\begin{eqnarray}\label{0.1}
\left|{}_2F_1(a,b;c;x)\right|\leq \frac{P_n\left(|a|,|b|,|c|\right)}{|(c)_{2n}|}\left|\frac{\Gamma(c+2n)}{\Gamma(b+n)\Gamma(c-b+n)}\right|,\hspace{3mm} 0\leq x<1,\\
\textup{whenever}\hspace{1mm}
\Re (c-a-b)>R_1, \Re(c-b)>R_2, \Re b>-n+\frac{1}{2}.
\end{eqnarray}
Now assume that $\Re(c-a-b)>R_1, \Re(c-b)>R_2,\Re b>-(n+1)+\frac{1}{2}$. By \ref{properties of hypergeometric function-2}, 
\begin{eqnarray*}
 c(c+1){}_2F_1(a,b;c;x)&=& c(c-a+1)_2F_1(a,b+1;c+2;x)\\&&+
a\left[c-(c-b)x\right]{}
 _2F_1(a+1,b+1;c+2;x), 0\leq x<1.
\end{eqnarray*} 
Note that we can use the induction hypothesis (\ref{0.1}) on the two hypergeometric function appeared on the right hand side. So we get (for $0\leq x<1$)
\begin{eqnarray*}
\left|{}_2F_1(a,b;c;x)\right|&\leq &\frac{|c|(|c|+|a|+1)}{|c||c+1|} \frac{P_n\left(|a|,|b+1|,|c+2|\right)}{|(c+2)_{2n}|}\left|\frac{\Gamma(c+2n+2)}{\Gamma(b+n+1)\Gamma(c-b+n+1)}\right|\\
&&+\frac{|a|(|c|+|c|+|b|)}{|c||c+1|} \frac{P_n\left(|a+1|,|b+1|,|c+2|\right)}{|(c+2)_{2n}|}\left|\frac{\Gamma(c+2n+2)}{\Gamma(b+n+1)\Gamma(c-b+n+1)}\right|.
\end{eqnarray*}
Since the coefficients of $P_n$ are all non-negative,
$$
P_n(|a|,|b+1|,|c+2|)\leq P_n(|a|,|b|+1,|c|+2)
$$
which can be written as a $2n$th degree polynomial in $|a|,|b|,|c|$ with all the coefficients are non-negative. The same is true for $P_n(|a+1|,|b+1|,|c+2|)$.
Again, $|c||(c+1)||(c +2)_{2n}|=|(c)_{2n+2}|$. Therefore, it follows that the lemma is true for $k=n+1$. 
\end{proof}

\begin{lemma}\label{gamma-1}
Let $a, b>0$ be fixed real number. Also fix $\delta>0$. Then
$$
\left|\frac{\Gamma(a+z)}{\Gamma(b+z)}\right|\asymp (1+|z|)^{a-b}\hspace{3mm}\textup{for all}\hspace{1mm}|\arg z|\leq\pi-\delta.
$$

\end{lemma}
\begin{proof}
Let $z$ be any complex number such that $|\arg z|\leq\pi-\delta$. By Starling formula we can say that
\begin{eqnarray*}
\left|\frac{\Gamma(a+z)}{\Gamma(b+z)}\right|&\asymp& \left|\frac{\sqrt{\frac{2\pi}{a+z}}(\frac{a+z}{e})^{a+z}}{\sqrt{\frac{2\pi}{b+z}}(\frac{b+z}{e})^{b+z}}\right|\\
&\asymp & \left|\frac{(a+z)^{a+z}}{(b+z)^{b+z}}\right|.
\end{eqnarray*}
Since $|(a+z)^a|=|a+z|^a\asymp(1+|z|)^a$ and $|(b+z)^b|\asymp (1+|z|)^b$, we have
 \begin{eqnarray*}
\left|\frac{\Gamma(a+z)}{\Gamma(b+z)}\right|
&\asymp & (1+|z|)^{a-b}\left|\frac{(a+z)^{z}}{(b+z)^{z}}\right|= (1+|z|)^{a-b}\left|\frac{(1+\frac{a}{z})^z}{(1+\frac{b}{z})^{z}}\right|.
\end{eqnarray*}
Now, writing $z=re^{i\theta}$, $|\theta|\leq\pi-\delta$,
$$
\lim_{r\rightarrow\infty}\left(1+\frac{a}{re^{i\theta}}\right)^{re^{i\theta}}=
\left[\lim_{r\rightarrow\infty}\left(1+\frac{ae^{-i\theta}}{r}\right)^{r}\right]^{e^{i\theta}}=e^a.
$$
Therefore
$$ 
\lim_{z\rightarrow\infty}\left(1+\frac{a}{z}\right)^{z}=e^a.
$$
Similarly
$$ 
\lim_{z\rightarrow\infty}\left(1+\frac{b}{z}\right)^{z}=e^b.
$$
Hence the proof follows.
\end{proof}

\begin{lemma}\label{estimate-c-function}
Let $a,  b, c>0$ be fixed real numbers and let  $\delta>0$ be fixed. Then
$$
\left|\frac{\Gamma(a+\frac z2) \Gamma(b + \frac z2)}{\Gamma(c+z)}\right|\asymp 2^{-\Re z}(1+|z|)^{a + b-c-\frac 12}\hspace{3mm}\textup{for all}\hspace{1mm}|\arg z|\leq\pi-\delta.
$$
\end{lemma}
\begin{proof}
 By Starling formula we have, for $|\arg z|\leq \pi-\delta$, 
\begin{eqnarray*}
\left|\frac{\Gamma(a+\frac{z}{2})\Gamma(b+\frac{z}{2})}{\Gamma(c+z)}\right|&\asymp& \left|\frac{\sqrt{\frac{2\pi}{a+\frac{z}{2}}}(\frac{a+\frac{z}{2}}{e})^{a+\frac{z}{2}}\sqrt{\frac{2\pi}{b+\frac{z}{2}}}(\frac{b+\frac{z}{2}}{e})^{b+\frac{z}{2}}}{\sqrt{\frac{2\pi}{c+z}}(\frac{c+z}{e})^{c+z}}\right|\\
&\asymp & (1+|z|)^{-\frac{1}{2}}\left|\frac{(a+\frac{z}{2})^{a+\frac{z}{2}}(b+\frac{z}{2})^{b+\frac{z}{2}}}{(c+z)^{c+z}}\right|\\
& \asymp & (1+|z|)^{a+b-c-\frac{1}{2}}\left|\frac{(a+\frac{z}{2})^{\frac{z}{2}}(b+ \frac{z}{2})^{\frac{z}{2}}}{(c+z)^{z}}\right|\\
&=& 2^{-\Re z}(1+|z|)^{a+b-c-\frac{1}{2}}\left|\frac{(2a+z)^{\frac{z}{2}}(2b+z)^{\frac{z}{2}}}{(c+z)^{z}}\right|\\
&=& 2^{-\Re z}(1+|z|)^{a+b-c-\frac{1}{2}}\frac{\left|(1+\frac{2a}{z})^{z}(1+\frac{2b}{z})^{z}\right|^{\frac{1}{2}}}{\left|(1+\frac{c}{z})^{z}\right|}.
\end{eqnarray*}
Now the proof can be completed as in the previous Lemma.
\end{proof}

\end{document}